\documentclass[leqno, 12pt]{amsart} 
\usepackage[mathscr]{eucal}
\usepackage{mathrsfs}
\usepackage{amscd}
\usepackage{amsfonts}
\usepackage{amsmath, amsthm, amssymb}
\usepackage{latexsym}
\usepackage{bm}
\usepackage[dvips]{graphics}
\usepackage{color,graphicx} 
\numberwithin{equation}{section}

\theoremstyle{plain} %
\newtheorem{theorem}{\indent Theorem}[section] %
\newtheorem{lemma}[theorem]{\indent Lemma}
\newtheorem{corollary}[theorem]{\indent Corollary}
\newtheorem{proposition}[theorem]{\indent Proposition}

\newtheorem{fact}[theorem]{\indent Fact}
\theoremstyle{definition} %
\newtheorem{definition}[theorem]{\indent Definition}
\newtheorem{remark}[theorem]{\indent Remark}
\newtheorem{example}[theorem]{\indent Example}

\newcommand{\trace}{\mathop{\mathrm{trace}}\nolimits}

\def\tr#1{\mathord{\mathopen{{\vphantom{#1}}^t}#1}} 
\def\rank{\mathop{\mathrm{rank}}\nolimits}
\def\ord{\mathop{\mathrm{ord}}\nolimits}

\def\C{{\mathbf{C}}}
\def\R{{\mathbf{R}}}
\def\H{{\mathbf{H}}}
\def\Z{{\mathbf{Z}}}
\def\N{{\mathbf{N}}}
\def\D{{\mathbf{D}}}
\def\L{{\mathbf{L}}}
\def\Pi{{\mathbf{P}}}
\def\Si{{\mathbf{S}}}
\def\E{{\mathcal{E}}}

\title[Lagrangian Gauss map]{Value distribution of the Gauss map of improper affine spheres}
\author[Y.~Kawakami]{Yu Kawakami}
\author[D.~Nakajo]{Daisuke Nakajo}

\thanks{The first author was partially supported by the Grants-in-Aid for Young Scientists (B) No. 21740053, 
Japan Society for the Promotion of Science and Global COE program (Kyushu university) 
``Education and Research Hub for Mathematics-for-Industry''. }

\address{
Graduate School of Science and Engineering, \endgraf
Yamaguchi university, \endgraf
Yamaguchi, 753-8512, Japan}
\email{y-kwkami@yamaguchi-u.ac.jp}

\address{Faculty of Mathematics, \endgraf
Kyushu university, \endgraf
744, Motooka, Nishiku,
Fukuoka-city,
819-0395, Japan}
\email{nakajo@math.kyushu-u.ac.jp}

\begin{document}

\maketitle

\begin{abstract}
We give the best possible upper bound for the number of exceptional values of the Lagrangian Gauss map of complete 
improper affine fronts in the affine three-space. We also obtain the sharp estimate for weakly complete case. 
As an application of this result, we provide a new and simple proof of the parametric affine Bernstein problem 
for improper affine spheres. Moreover, we get the same estimate for the ratio of canonical forms of weakly complete 
flat fronts in hyperbolic three-space. 
\end{abstract}

\section*{Introduction}
The study of improper affine spheres has been related to various subjects in geometry and analysis. 
In fact, improper affine spheres in the affine three-space ${\R}^{3}$ are locally graphs of solutions of 
the Monge-Amp\`ere equation $\det{({\nabla}^{2} f)}=1$, and Calabi \cite{Ca2} proved that there exists a local correspondence 
between solutions of the equation of improper affine spheres in ${\R}^{3}$ and solutions of 
the equation of minimal surfaces in Euclidean three-space. Recently, Mart\'inez \cite{Ma1} discovered 
the correspondence between improper affine spheres and smooth special Lagrangian immersions in 
the complex two-space ${\C}^{2}$. Moreover, from the viewpoint of this correspondence, he introduced the notion 
of improper affine maps, that is, a class of (locally strongly convex) improper affine spheres with some admissible 
singularities and gave a holomorphic representation formula for them. Later, the second author \cite{Na},  
Umehara and Yamada \cite{UY2} showed that an improper affine map is a front in ${\R}^{3}$, and hence we 
call this class improper affine fronts in this paper. 
Mart\'inez \cite{Ma1} also defined the Lagrangian Gauss map of improper affine fronts in ${\R}^{3}$ 
and obtained the characterization of a complete (in the sense of \cite{KUY2, Ma1}, see also Section 1 of this paper) 
improper affine front whose Lagrangian Gauss map is constant. We note that the second author \cite{Na} 
constructed a representation formula for indefinite improper affine spheres with some admissible singularities. 

On the other hand, the study of value distribution property of the Gauss map of complete minimal surfaces in Euclidean 
three-space has accomplished many significant results. 
This study is a generalization of the classical Bernstein theorem \cite{Be} and initiated by Osserman \cite{Os3, Os1, Os2}. 
Fujimoto \cite{Fu1, Fu2, Fu3} showed that the best possible upper bound for the number $D_{g}$ of exceptional 
values of the Gauss map $g$ of complete nonflat minimal surfaces in Euclidean three-space is ``four''. 
Ros \cite{Ros} gave a different proof of this result. Moreover, Osserman \cite{Os1, Os2} proved that the Gauss map of 
a nonflat algebraic minimal surface 
can omit at most three values (by an algebraic minimal surface, we mean a complete minimal surface with finite total curvature). 
Recently, the first author, Kobayashi and Miyaoka \cite{KKM} gave an effective 
ramification estimate for the Gauss map of a wider class of complete minimal surfaces that includes algebraic minimal 
surfaces (this class is called the pseudo-algebraic minimal surfaces). It also provided new proofs of the Fujimoto and the Osserman 
theorems in this class and revealed the geometric meaning behind them. 
The first author obtained the same estimate for the hyperbolic Gauss map of pseudo-algebraic constant mean curvature one 
surfaces in hyperbolic three-space ${\H}^{3}$ \cite{Ka3}. These estimates correspond to the defect relation in Nevanlinna 
theory (\cite{JR}, \cite{Ko}, \cite{NO} and \cite{Ru}). 

The purpose of this paper is to study value distribution property of the Lagrangian Gauss map of improper affine fronts in 
${\R}^{3}$. The organization of this paper is as follows: In Section 1, we recall some definitions and 
basic facts about improper affine fronts in ${\R}^{3}$ which are used throughout this paper. 
We review, in particular, the definitions of completeness in the sense of \cite{KUY2, Ma1} and weakly completeness in the sense of \cite{UY2}. 
In Section 2, we give the upper bound for the totally ramified value number ${\delta}_{\nu}$ of 
the Lagrangian Gauss map $\nu$ of complete improper affine fronts in ${\R}^{3}$ (Theorem \ref{Thm2-1}). 
This estimate is effective in the sense that the upper bound which we obtained 
is described in terms of geometric invariants and sharp for some topological cases. Moreover, as a corollary of 
this estimate, we also obtain the best possible upper bound for the number $D_{\nu}$ of exceptional values of the Lagrangian 
Gauss map in this class (Corollary \ref{Cor2-1}).  We note that this class corresponds to that of algebraic minimal 
surfaces in Euclidean three-space. 
In Section 3, by applying the Fujimoto argument, we give the optimal estimate for $D_{\nu}$ of 
weakly complete improper affine fronts in ${\R}^{3}$ (Theorem \ref{The3-2}). 
We note that the best possible upper bound for $D_{g}$ of complete minimal surfaces obtained by Fujimoto is ``four'', 
but the best possible upper bound for $D_{\nu}$ of this class is ``three''. 
As an application of this estimate, from the viewpoint of the value distribution property, 
we provide a new and simple proof of the well-known result (\cite{Ca1}, \cite{Jo}) 
that any affine complete improper affine sphere must be an elliptic paraboloid (Corollary \ref{Cor2-2}). 
This result is the special case of the parametric affine Bernstein problem of affine maximal surfaces, which states that 
any affine complete affine maximal surface must be an elliptic paraboloid (\cite{Ca3}, \cite{LSZ}, and \cite{TW}). 
In Section 4, after reviewing some definitions and fundamental properties on flat fronts in ${\H}^{3}$, 
we study the value distribution of the ratio of canonical forms of weakly complete flat fronts in ${\H}^{3}$. 
Flat surfaces (resp. fronts) in ${\H}^{3}$ are closely related to improper affine spheres (resp. fronts) in  
${\R}^{3}$ (See \cite{Ma2} and also \cite{IM}). Indeed, we show that the ratio of canonical forms of weakly complete flat fronts in ${\H}^{3}$ 
have some properties similar to the Lagrangian Gauss map of weakly complete improper affine fronts in ${\R}^{3}$ 
(Propositions \ref{Pro3-1} and \ref{prop4-4}, Theorems \ref{The3-2} and \ref{Thm4-5}, Corollaries \ref{Cor2-2} and \ref{Cor4-6}). 
By Corollaries \ref{Cor2-2} and \ref{Cor4-6}, we can prove that the uniqueness theorems of complete surfaces (these results are called the 
parametric Bernstein type theorems) for these classes follow from the Liouville property, that is, the boundedness of their Gauss maps. 

Finally, the authors would like to particularly thank to Professors Wayne Rossman, Masaaki Umehara and 
Kotaro Yamada for their useful advice. The authors also thank to Professors Ryoichi Kobayashi, 
Masatoshi Kokubu, Reiko Miyaoka, Junjiro Noguchi and Yoshihiro Ohnita for their encouragement of our study.

\section{Preliminaries}
We first briefly recall some definitions and basic facts about affine differential geometry. 
Details can be found, for instance, in  \cite{LSZ} and \cite{NS}. Let $\Sigma$ be an oriented two-manifold, and $(\psi, \xi)$ a pair of an immersion $\psi\colon \Sigma\to {\R}^{3}$ 
into the affine three-space ${\R}^{3}$ and a vector field $\xi$ on $\Sigma$ along $\psi$ which is transversal to 
${\psi}_{\ast}(T\Sigma)$. Then the Gauss-Weingarten equations of $(\psi, \xi)$ are as follows: 
\begin{eqnarray}\label{eq1-1}
\left\{
\begin{array}{l}
D_{X}{\psi}_{\ast}Y = {\psi}_{\ast}({\nabla}_{X}Y)+g(X, Y)\xi\,, \\
D_{X} \xi           = -{\psi}_{\ast}(SX)+\tau (X)\xi\, , 
\end{array}
\right.
\end{eqnarray} 
where $D$ is the standard flat connection on ${\R}^{3}$. Here, $g$ is called the {\it affine metric} (or {\it Blaschke metric}) of the pair 
$(f, \xi)$. Indeed, we can easily show that the rank of $g$ is invariant under the change of the transversal vector 
field $\xi$. When $g$ is positive definite, we call $\psi$ a {\it locally strongly convex immersion}. 
From now on, we only consider the locally strongly convex case. 
Given an immersion $\psi\colon \Sigma\to {\R}^{3}$, we can uniquely choose the transversal vector field $\xi$ 
which satisfies the following conditions: 
\begin{enumerate}
\item[(i)] $\tau \equiv 0$ (or equivalently $D_{X}\xi \in {\psi}_{\ast}(T\Sigma)$ for all $X\in \mathfrak{X}(\Sigma))$\, ,
\item[(ii)] ${\text{vol}}_{g}(X_{1}, X_{2})=\det{({\psi}_{\ast}X_{1}, {\psi}_{\ast}X_{2}, \xi)}$ for all $X_{1}, X_{2}\in \mathfrak{X}(\Sigma)$\, ,
\end{enumerate}
where ${\text{vol}}_{g}$ is the volume form of the Riemannian metric $g$ and $\det$ is the standard volume element 
of ${\R}^{3}$. The transversal vector field $\xi$ which satisfies the two conditions above is called an 
{\it affine normal} (or {\it Blaschke normal}), and a pair $(\psi, \xi)$ of an immersion and its affine normal 
is called a {\it Blaschke immersion}. A Blaschke immersion $(f, \xi)$ with $S=0$ in (\ref{eq1-1}) is called 
an {\it improper affine sphere}. In this case, the transversal vector field $\xi$ is constant  because $\tau \equiv 0$. 
Thus a transversal vector field $\xi$ of an improper affine sphere is given by $\xi=(0, 0, 1)$ 
after a suitable affine transformation of ${\R}^{3}$. 
The {\it conormal map} $N\colon \Sigma\to ({\R}^{3})^{\ast}$ into the dual space of the affine three-space $({\R}^{3})^{\ast}$ 
for a given Blaschke immersion $(f, \xi)$ 
is defined as the immersion which satisfy the following conditions: 
\begin{enumerate}
\item[(i)] $N(f_{\ast}X)=0$ for all $X\in \mathfrak{X}(\Sigma)$\, ,
\item[(ii)] $N(\xi)=1$\,.
\end{enumerate}
For an improper affine sphere with affine normal $(0, 0, 1)$, we can write $N=(n, 1)$ with a smooth map 
$n\colon \Sigma\to {\R}^{2}$. 

Let ${\C}^{2}$ denotes the complex two-space with the complex coordinates $\zeta=({\zeta}_{1}, {\zeta}_{2})$, 
where $\zeta =x+\sqrt{-1}y$ $(x, y\in {\R}^{2})$. We consider the standard metric $g'$, the symplectic form ${\omega}'$, 
and the complex two-form $\Omega'$ given by
\begin{eqnarray}
g'&=&|d{\zeta}_{1}|^{2}+|d{\zeta}_{2}|^{2}\, , \nonumber \\ 
\omega'&=& \dfrac{\sqrt{-1}}{2}(d{\zeta}_{1}\wedge d{\bar{\zeta}}_{1}+d{\zeta}_{2}\wedge d{\bar{\zeta}}_{2})\, , \nonumber \\
\Omega'&=&d{\zeta}_{1}\wedge d{\zeta}_{2}\, . \nonumber 
\end{eqnarray}

Let $L\colon \Sigma\to {\C}^{2}$ be an special Lagrangian immersion with respect to the calibration $\Re (\sqrt{-1}{\Omega}')$. 
As in \cite{HL}, $L$ can be characterized as an immersion in ${\C}^{2}$ satisfying 
\[
\omega'|_{L(\Sigma)}\equiv 0, \quad \Im (\sqrt{-1}{\Omega}'|_{L(\Sigma)})\equiv 0\,,
\]
where $\Re$ and $\Im$ represent real and imaginary part, respectively. 

Then there exists the following correspondence between improper affine spheres in ${\R}^{3}$ and some nondegenerate 
special Lagrangian immersions in ${\C}^{2}$. 
\begin{fact}{\cite[Theorem 1]{Ma1}}\label{FACT-A}
Let $\psi=(x, \varphi)\colon \Sigma\to {\R}^{3}={\R}^{2}\times {\R}$ be an improper affine sphere with the conormal map 
$N=(n, 1)$. The map $L_{\psi}\colon \Sigma\to {\C}^{2}$ given by 
\[
L_{\psi}:=x+\sqrt{-1}n
\]
is an special Lagrangian immersion such that
\begin{enumerate}
\item[(i)] The induced metric $d{\tau}^{2}:=\langle dx, dx \rangle + \langle dn, dn \rangle$ is conformal to the affine metric $g$ of $\psi$\,,  
\item[(ii)] The metric $ds^{2}:=\langle dx, dx \rangle$ is a nondegenerate flat metric\,, 
\end{enumerate}
where $\langle \cdot\, , \cdot \rangle$ denotes the standard inner product in ${\R}^{2}$. 
\end{fact}

\begin{fact}{\cite[Theorem 2]{Ma1}}\label{FACT-B}
Let $L_{\psi}=x+\sqrt{-1}n \colon \Sigma\to {\C}^{2}$ be a special Lagrangian immersion such that $ds^{2}:=\langle dx, dx \rangle$ is nondegenarate. 
Then 
\[
\psi =\biggl(x, -\int \langle n, dx \rangle  \biggr)
\]
is an improper affine sphere which is well-defined if and only if $\int_{c} \langle n, dx \rangle = 0$ for any loop $c$ on $\Sigma$. 
\end{fact}

Next, using the notations defined as above, we introduce the notion of improper affine fronts, 
which is a generalization of improper affine spheres with some admissible singularities. 
\begin{definition}[{\cite[Definition 1]{Ma1}}]
A map $\psi=(x, \varphi)\colon \Sigma\to {\R}^{3}={\R}^{2}\times {\R}$ is called an {\it improper affine front} 
if $\psi$ is expressed as 
\[
\psi =\biggl(x, -\int \langle n, dx \rangle  \biggr)
\]
by a special Lagrangian immersion $L_{\psi}=x+\sqrt{-1}n\colon \Sigma \to {\C}^{2}$, 
where $\langle \cdot\, , \cdot \rangle$ denotes the standard inner product in ${\R}^{2}$. 
Nonregular points of $\psi$ correspond with degenerate points of 
$ds^{2}:=\langle dx, dx \rangle$. We call $ds^{2}$ the {\it flat fundamental form} of $\psi$. 
\end{definition}

From Facts \ref{FACT-A} and \ref{FACT-B}, at the nondegenerate points of $ds^{2}$, the induced metric 
$d{\tau}^{2}:=\langle dx, dx \rangle + \langle dn, dn \rangle$ is conformal to the affine metric 
$g:=-\langle dx, dn \rangle$. 

For any improper affine front $\psi\colon \Sigma \to{\R}^{3}$, considering the conformal structure given by the induced metric $d{\tau}^{2}$ 
of its associated special Lagrangian immersion $L_{\psi}$, we regard $\Sigma$ as a Riemann surface. 

Since every special Lagrangian immersion in ${\C}^{2}$ is a holomorphic curve with respect to the complex coordinates $\zeta=({\zeta}_{1}, {\zeta}_{2})$ 
(see \cite{Joy}), we see that there exists a holomorphic regular curve $\alpha\colon \Sigma \to {\C}^{2}$, 
$\alpha :=(F, G)$, such that if we identify vectors of ${\R}^{2}$ with complex numbers in the standard way: 
\[
(r, s)=r+\sqrt{-1}s, \quad r,\, s\in \R
\]
then we can write 
\begin{equation}\label{eq1-2}
x=G+\bar{F}, \quad n=\bar{F}-G
\end{equation}
and since the inner product of two vectors ${\zeta}_{i}=r_{i}+\sqrt{-1}s_{i}$ $(i=1, 2)$ is given by 
$\langle {\zeta}_{1}, {\zeta}_{2} \rangle = \Re({{\zeta}_{1}\bar{{\zeta}_{2}}})$, 
then the flat fundamental form $ds^{2}$, the induced metric $d{\tau}^{2}$ and the affine metric $g$ are given, 
respectively,  by
\begin{eqnarray}\label{eq1-3}
ds^{2}&=&|dF+dG|^{2}=|dF|^{2}+|dG|^{2}+dGdF+\overline{dFdG}\,, \nonumber \\
d{\tau}^{2}&=&2(|dF|^{2}+|dG|^{2})\,, \\
g&=&|dG|^{2}-|dF|^{2}\,. \nonumber
\end{eqnarray}
The nontrivial part of the Gauss map of $L_{\psi}\colon \Sigma\to {\C}^{2}\cong {\R}^{4}$ (see \cite{CM}) is the meromorphic function  
$\nu\colon \Sigma \to \C\cup \{\infty\}$ given by 
\begin{equation}\label{eq1-4}
\nu :=\frac{dF}{dG}
\end{equation}
which is called the {\it Lagrangian Gauss map} of $\psi$. 

Mart\'inez \cite{Ma1} gave the following representation formula for improper affine fronts 
in terms of two holomorphic functions. 
This generalized a formula in \cite{FMM}. 

\begin{fact}[{\cite[Theorem 3]{Ma1}}] \label{IA-fact1}
Let $\psi =(x, \varphi) \colon \Sigma \to {\R}^{3}={\C}\times{\R}$ be an improper affine front. 
Then there exists a holomorphic regular curve $\alpha :=(F, G)\colon \Sigma \to {\C}^{2}$ 
such that 
\begin{equation}\label{IA-Wrep}
\psi:= \biggl(G+\bar{F}, \frac{|G|^{2}-|F|^{2}}{2}+\Re\biggl(GF-\int FdG \biggr) \biggr)\,.
\end{equation}
Here, the conormal map of $\psi$ becomes 
\[
N=(\bar{F}-G, 1)\,.
\]
Conversely, given a Riemann surface $\Sigma$ and a holomorphic regular curve $\alpha :=(F, G)\colon \Sigma \to {\C}^{2}$, 
then (\ref{IA-Wrep}) gives an improper affine front which is well-defined if and only if $\Re \int_{c} FdG= 0$ for any loop 
$c$ in $\Sigma$. 
\end{fact}

We call the pair $(F, G)$ the {\it Weierstrass data} of $\psi$. 
Note that the singular points of $\psi$ correspond with the points where $|dF|=|dG|$, that is, $|\nu|=1$ 
(\cite{Ma1}, see also \cite{Na}). 

An improper affine front $\psi\colon \Sigma \to {\R}^{3}$ is said to be {\it complete} 
if there exists a symmetric two-tensor $T$ such that $T=0$ outside a compact set $C\subset \Sigma$ 
and $ds^{2}+T$ is a complete Riemannian metric on $\Sigma$, where $ds^{2}$ is the flat fundamental form 
of $\psi$. This definition is similar to the definition of completeness for fronts \cite{KUY2}. 

\begin{fact}\label{IA-fact2}
A complete improper affine front $\psi\colon \Sigma\to {\R}^{3}$ satisfies the following two conditions: 
\begin{enumerate}
\item[(i)] $\Sigma$ is biholomorphic to ${\overline{\Sigma}}_{\gamma}\backslash \{p_{1},\ldots, p_{k}\}$, 
where ${\overline{\Sigma}}_{\gamma}$ is a closed Riemann surface of genus $\gamma$ and 
$p_{j}\in {\overline{\Sigma}}_{\gamma}$ $(j=1,\ldots, k)$ \cite{Hu}. 
\item[(ii)] The Weierstrass data $(F, G)$ of $\psi$ can be extended meromorphically to ${\overline{\Sigma}}_{\gamma}$. 
In particular, its Lagrangian Gauss map can also be a meromorphic function on ${\overline{\Sigma}}_{\gamma}$ 
\cite{Ma1}. 
\end{enumerate}
\end{fact}

Each puncture point $p_{j}$ $(j=1,\ldots, k)$ is called an {\it end} of $\psi$. 
On the other hand, an improper affine front is said to be {\it weakly complete} if the induced metric $d{\tau}^{2}$ as in (\ref{eq1-3}) 
is complete. Note that the universal cover of a weakly complete improper affine front is also weakly complete, 
but completeness is not preserved when lifting to the universal cover. 
The relationship between completeness and weakly completeness in this class is as follows: 

\begin{fact}[{\cite[Remark 4]{UY2}}]\label{IA-fact3}
An improper affine front in ${\R}^{3}$ is complete if and only if it is weakly complete, 
the singular set is compact and all ends are biholomorphic to a punctured disk. 
\end{fact}

Finally, we give two examples in \cite[Section 4]{Ma1} which play important roles in the following sections. 

\begin{example}[Elliptic paraboloids]\label{IA-Ex1}
An elliptic paraboloid can be obtained by taking $\Sigma =\C$ and Weierstrass data $(cz, z)$, where $c$ is constant. 
It is complete, and its Lagrangian Gauss map is constant. Note that, if $|c|=1$, then an improper affine front constructed 
from this data is a line in ${\R}^{2}$. 
\end{example}

\begin{example}[Rotational improper affine fronts]\label{IA-Ex2}
A rotational improper affine front is obtained by considering $\Sigma =\C\backslash \{0\}$ and 
Weierstrass data $(z, \pm r^{2}/z)$, where $r\in {\R}\backslash \{0\}$. It is complete 
and its Lagrangian Gauss map $\nu= \mp z^{2}/r^{2}$. In particular, $\nu$ omits two values, $0$, $\infty$. 
\end{example}

\section{A ramification estimate for the Lagrangian Gauss map of complete improper affine fronts}

We first recall the definion of the totally ramified value number of a meromorphic function. 

\begin{definition}[\cite{Ne}]
Let $\Sigma$ be a Riemann surface and $h$ a meromorphic function on $\Sigma$. 
We call $b\in \C\cup\{\infty\}$ a {\it totally ramified value} of $h$ 
when $h$ branches at any inverse image of $b$. 
We regard exceptional values also as totally ramified values, here we call a point of $\C\cup\{\infty\}\backslash h(\Sigma)$ 
an exceptional value of $h$. 
Let $\{a_1,\dots,a_{r_0},b_1,\dots,b_{l_0}\}\subset \C\cup \{\infty\}$ be the set of 
totally ramified values of $h$, where $a_j$'s are exceptional 
values. For each $a_j$, set $m_j=\infty$, and for each $b_j$, 
define $m_j$ to be the minimum of the multiplicities of $h$ at 
points $h^{-1}(b_j)$. Then we have $m_j\geq 2$.
We call 
\[
\delta_{h}=\sum_{a_j,b_j}\biggl(1-\dfrac1{m_j}\biggr)=r_0+\sum_{j=1}^{l_0}\biggl(1-\dfrac1{m_j}\biggr)
\]
the {\it totally ramified value number} of $h$. 
\end{definition}

Because the Lagrangian Gauss map $\nu$ of an improper affine front $\psi\colon \Sigma\to {\R}^{3}$ is 
a meromorphic function on $\Sigma$, we can consider the totally ramified value number ${\delta}_{\nu}$ of $\nu$. 
By virtue of Fact \ref{IA-fact2}, we regard $\Sigma$ as a punctured Riemann surface 
${\overline{\Sigma}}_{\gamma}\backslash \{p_{1},\ldots,p_{k}\}$, where ${\overline{\Sigma}}_{\gamma}$ is 
a closed Riemann surface of genus $\gamma$ and 
$p_{j}\in {\overline{\Sigma}}_{\gamma}$ $(j=1,\ldots, k)$. 
Then we give the upper bound for ${\delta}_{\nu}$ of complete improper affine fronts in ${\R}^{3}$. 
Here, we denote by $D_{\nu}$ the number of exceptional values 
of $\nu$. By definition, it follows immediately that $D_{\nu}\leq {\delta}_{\nu}$. 
\begin{theorem}\label{Thm2-1}
Let $\psi\colon \Sigma={\overline{\Sigma}}_{\gamma}\backslash \{p_{1},\ldots,p_{k}\}\to {\R}^{3}$ 
be a complete improper affine front and $\nu\colon \Sigma\to \C\cup \{\infty\}$ the Lagrangian Gauss map 
of $\psi$. 
Suppose that $\nu$ is nonconstant and $d$ is the degree of $\nu$ considered as a map on ${\overline{\Sigma}}_{\gamma}$.    Then we have 
\begin{equation}\label{IA-estimate}
D_{\nu}\leq {\delta}_{\nu}\leq 2+\dfrac{2}{R}, \qquad \dfrac{1}{R}:=\dfrac{\gamma-1+k/2}{d}<\dfrac{1}{2}\,.
\end{equation}
In particular, $D_{\nu}\leq {\delta}_{\nu}< 3$\,.
\end{theorem}
\begin{remark}
The geometric meaning of ``$2$'' in the upper bound of (\ref{IA-estimate}) is the Euler number of the Riemann sphere. 
The geometric meaning of the ratio $R$ is given in \cite[Section 6]{KKM}. 
\end{remark}
\begin{proof}
By Fact \ref{IA-fact3}, if $ds^{2}$ is complete, then $d{\tau}^{2}$ is a complete Riemannian metric. Then 
the metric $d{\tau}^{2}$ is represented as 
\begin{equation}\label{eq2-1-1}
{d\tau}^{2} =2(|dF|^{2}+|dG|^{2})=2\biggl( 1+\biggl|\frac{dF}{dG}{\biggr|}^{2} \biggr)|dG|^{2}=2(1+|\nu|^{2})|dG|^{2}\,.
\end{equation} 
Since $d{\tau}^{2}$ is nondegenerate on $\Sigma$, the poles of $\nu$ of order $k$ coincide exactly with 
the zeros of $dG$ of order $k$.  By the completeness of $d{\tau}^{2}$, 
$dG$ has a pole of order ${\mu}_{j}\geq 1$ at $p_{j}$ \cite{Os2}. Moreover we show that 
${\mu}_{j}\geq 2$ for each $p_{j}$ because $G$ is single-valued on ${\overline{\Sigma}}_{\gamma}$. 
Applying the Riemann-Roch theorem to $dG$ on ${\overline{\Sigma}}_{\gamma}$, we have 
\begin{equation}\label{eq2-1-2}
d-\displaystyle \sum_{j=1}^{k}{\mu}_{j}=2\gamma -2\,.
\end{equation}
Thus we get 
\begin{equation}\label{eq2-1-3}
d=2\gamma -2+\displaystyle \sum_{j=1}^{k}{\mu}_{j}\geq 2(\gamma -1+ k)> 2\biggl( \gamma -1+ \frac{k}{2} \biggr) \, ,
\end{equation}
and 
\begin{equation}\label{eq2-1-4}
\frac{1}{R} <\frac{1}{2}\,.
\end{equation} 

Assume that $\nu$ omits $r_{0}=D_{\nu}$ values. Let $n_{0}$ be the sum of the branching orders at the image of 
these exceptional values of $\nu$. Then we have 
\begin{equation}\label{eq2-1-5}
k\geq dr_{0}-n_{0}\,.
\end{equation}
Let $b_{1}, \ldots, b_{l_{0}}$ be the totally ramified values which are not exceptional values 
and $n_{r}$ the sum of branching order at the inverse image of $b_{i}$ $(i=1, \ldots, l_{0})$ of $\nu$. 
For each $b_{i}$, we denote 
\[
m_i=\text{min}_{{\nu}^{-1}(b_i)}\{\text{multiplicity of }\nu(z)=b_i\}\,,
\]
then the number of points in the inverse image ${\nu}^{-1}(b_{i})$ is less than or equal to $d/m_{i}$. 
Thus we get 
\begin{equation}\label{eq2-1-6}
dl_{0}-n_{r}\leq \displaystyle \sum_{i=1}^{l_{0}}\frac{d}{m_{i}}\, .
\end{equation} 
This implies 
\begin{equation}\label{eq2-1-7}
l_{0}- \displaystyle \sum_{i=1}^{l_{0}}\frac{1}{m_{i}} \leq \frac{n_{r}}{d}\, .
\end{equation}
Let $n_{\nu}$ be the total branching order of $\nu$ on ${\overline{\Sigma}}_{\gamma}$. 
Then applying the Riemann-Hurwitz formula to the meromorphic function $\nu$ on ${\overline{\Sigma}}_{\gamma}$, 
we have 
\begin{equation}\label{eq2-1-8}
n_{\nu}=2(d+\gamma -1)\,.
\end{equation} 
Hence we obtain 
\[
{\delta}_{\nu}=r_0+\sum_{j=1}^{l_0}\biggl(1-\dfrac1{m_j}\biggr)\leq \frac{n_{0}+k}{d} + \frac{n_{r}}{d}
\leq \frac{n_{\nu}+k}{d}=2+\dfrac{2}{R}\,.
\]
\end{proof}

The system of inequalities (\ref{IA-estimate}) is sharp in the following cases: 

(i) When $(\gamma, k, d)=(0, 1, n)$ $(n\in \N)$, we have 
\[
{\delta}_{\nu}\leq 2-\frac{1}{n}\,.
\] 
In this case, we can set $\Sigma =\C$. Since $\Sigma$ is simply connected, we have no period condition. 
We define a Weierstrass data on $\Sigma$, by 
\begin{equation}\label{w-data1}
(F, G)=\biggl(\frac{z^{n+1}}{n+1},\, z \biggr). 
\end{equation}
By Fact \ref{IA-fact1}, we can construct a complete improper affine front $\psi\colon \Sigma\to {\R}^{3}$ whose Weierstrass data 
is (\ref{w-data1}). In particular, its Lagrangian Gauss map $\nu$ has $\delta_{\nu}=2-(1/n)$. 
In fact, $\nu= z^{n}$, and it has one exceptional value and another totally ramified value of multiplicity $n$ at $z=0$. 
Thus (\ref{IA-estimate}) is sharp in this case. 

(ii) When $(\gamma, k, d)=(0, 2, 2)$, we have 
\[
D_{\nu}\leq {\delta}_{\nu}\leq 2\,.
\] 
In this case, we can set $\Sigma =\C\backslash \{0\}$. On the other hand, a rotational improper affine front 
(Example \ref{IA-Ex2}) has $D_{\nu}={\delta}_{\nu}=2$. 
Thus (\ref{IA-estimate}) is also sharp in this case. 

As a corollary of Theorem \ref{Thm2-1}, we obtain the maximal number of exceptional 
values of the Lagrangian Gauss map of complete improper affine fronts in ${\R}^{3}$. 
\begin{corollary}\label{Cor2-1}
Let $\psi$ be a complete improper affine front in ${\R}^{3}$. If its Lagrangian Gauss map $\nu$ is nonconstant, 
then $\nu$ can omit at most two values. 
\end{corollary}
The number ``two'' is sharp, because the Lagrangian Gauss map of a rotational 
improper affine front (Example \ref{IA-Ex2}) omits two values. Hence we provide the best possible upper bound for $D_{\nu}$ in complete case. 

\section{The maximal number of exceptional values of the Lagrangian Gauss map of weakly complete improper affine fronts} 
In this section, we study the value distribution of the Lagrangian Gauss map of weakly complete 
improper affine fronts in ${\R}^{3}$. We begin to consider the case where the Lagrangian Gauss map is constant. 
\begin{proposition}\label{Pro3-1}
Let $\psi\colon \Sigma \to {\R}^{3}$ be a weakly complete improper affine front. 
If its Lagrangian Gauss map $\nu$ is constant, then $\psi$ is an elliptic paraboloid. 
\end{proposition} 
\begin{proof}
Since the metric $d{\tau}^{2}$ is represented as (\ref{eq2-1-1}), if $\nu$ is constant,  
then the Gaussian curvature $K_{d{\tau}^{2}}$ of $d{\tau}^{2}$ vanishes identically on $\Sigma$. 
By the Huber theorem, $\Sigma$ is a closed Riemann surface of genus $\gamma$ with $k$ points removed, 
that is, $\Sigma={\overline{\Sigma}}_{\gamma}\backslash \{p_{1},\ldots,p_{k}\}$. 
Moreover we obtain the formula (see \cite[Corollary 1]{Fa} or \cite{Sh})
\[
\frac{1}{2\pi}\int_{\Sigma}(-K_{d\tau^{2}})dA=-\chi(\overline{\Sigma}_{\gamma})-
\sum_{j=1}^{k}{\ord}_{p_{j}}(d\tau^{2}), 
\] 
where $dA$ denotes the area element of $d{\tau}^{2}$ and $\chi(\overline{\Sigma}_{\gamma})$ 
the Euler number of $\overline{\Sigma}_{\gamma}$. 
Since the metric $d{\tau}^{2}$ is complete, 
${\ord}_{p_{j}}d{\tau}^{2}\leq -1$ holds for each end $p_{j}$. Thus if $\nu$ is constant, then we get $\gamma=0$ and 
\begin{equation}
\sum_{j=1}^{k}{\ord}_{p_{j}}(d\tau^{2})=-2\,.
\end{equation}
Since $d{\tau}^{2}$ is well-defined on $\Sigma$, we need to consider the following two cases: 
\begin{enumerate}
\item[(a)] The improper affine front $\psi$ has two ends $p$ and $q$, and ${\ord}_{p}{d\tau^{2}}={\ord}_{q}{d\tau^{2}}=-1$,
\item[(b)] The improper affine front $\psi$ has one end $p$, and ${\ord}_{p}{d\tau^{2}}=-2$\,. 
\end{enumerate}
In this class, the case (a) cannot occur because $F$ and $G$ are single-valued on $\Sigma$. 
Thus we have only to consider the case (b). Then we may assume that $p=\infty$ after a suitable M\"obius transformation 
of the Riemann sphere $\overline{\Sigma}_{0}$. Since $\nu$ is constant, $dF$ and $dG$ are well-defined on $\overline{\Sigma}_{0}$, 
and it holds that 
\[
{\ord}_{\infty}{dF}={\ord}_{\infty}{dG}=-2\,.
\]
We thus have $dF=cdz$ and $dG=dz$, that is, $F(z)=cz$ and $G(z)=z$ for some constant $c$. 
Therefore the result follows from Example \ref{IA-Ex1}. 
\end{proof}

We next give the best possible upper bound for the number of exceptional values of the Lagrangian Gauss map 
of weakly complete improper affine fronts in ${\R}^{3}$. 
\begin{theorem}\label{The3-2}
Let $\psi$ be a weakly complete improper affine front in ${\R}^{3}$. 
If its Lagrangian Gauss map $\nu$ is nonconstant, then $\nu$ can omit at most three values. 
\end{theorem}
The number ``three'' is sharp because there exist the following examples. 
\begin{example}[Improper affine fronts of Voss type]\label{Ex-Voss}
We consider the Lagrangian Gauss map $\nu$ and the holomorphic one-form $dG$ on $\Sigma ={\C}\backslash \{a_{1}, a_{2}\}$ 
for distinct points $a_{1}, a_{2}\in \C$, by
\begin{equation}\label{eq-Voss}
\nu=z,\quad dG=\dfrac{dz}{\prod_{j}(z-a_{j}) }\,.
\end{equation}
Since $F$ and $G$ are not well-defined on $\Sigma$, we obtain an improper affine front $\psi\colon \D\to {\R}^{3}$ on the universal covering disk $\D$ 
of $\Sigma$. Since the metric $d{\tau}^{2}$ is complete, we can get a weakly complete improper affine front whose 
Lagrangian Gauss map omits three values, $a_{1}$, $a_{2}$ and $\infty$.  
\end{example}


Before proceeding to the proof of Theorem \ref{The3-2}, we recall two function-theoretical lemmas. 
For two distinct values $\alpha$, $\beta\in \C\cup\{\infty\}$, we set 
\[
|\alpha, \beta|:=\dfrac{|\alpha -\beta|}{\sqrt{1+|\alpha|^{2}}\sqrt{1+|\beta|^{2}}}
\] 
if $\alpha\not= \infty$ and $\beta\not= 0$, and $|\alpha, \infty|=|\infty, \alpha|:=1/\sqrt{1+|\alpha|^{2}}$. 
Note that, if we take $v_{1}$, $v_{2}\in {\Si}^{2}$ with $\alpha =\varpi (v_{1})$ and $\beta =\varpi (v_{2})$, 
we have that $|\alpha, \beta|$ is a half of the chordal distance between $v_{1}$ and $v_{2}$, where $\varpi$ denotes 
the stereographic projection of ${\Si}^{2}$ onto $\C\cup\{\infty\}$. 

\begin{lemma}[{\cite[(8.12) in page 136]{Fu3}}] \label{Lem3-1}
Let $\nu$ be a nonconstant meromorphic function on ${\Delta}_{R} =\{z\in \C\, ;\, |z|<R\}$ $(0<R\leq +\infty)$ 
which omits $q$ values ${\alpha}_{1}, \ldots, {\alpha}_{q}$.  
If $q>2$, then for each positive $\eta$ with $\eta < (q-2)/q$, there exists some positive constant $C>0$ such that 
\begin{equation}\label{Fuji-1}
\dfrac{|{\nu}'|}{(1+|\nu|^{2}){\prod}_{j=1}^{q}|\nu ,{\alpha}_{j}|^{1-\eta}}
\leq C\dfrac{R}{R^{2}-|z|^{2}}\, .
\end{equation}
\end{lemma}

\begin{lemma}[{\cite[Lemma 1.6.7]{Fu2}}]\label{Lem3-2}
Let $d{\sigma}^{2}$ be a conformal flat metric on an open Riemann surface $\Sigma$. 
Then, for each point $p\in \Sigma$, there exists a local diffeomorphism $\Psi$ of a 
disk ${\Delta}_{R_{0}}=\{z\in \C\, ;\, |z|<{R}_{0}\}$ $(0<{R}_{0}\leq +\infty)$ onto an open neighborhood 
of $p$ with $\Psi (0)=p$ such that $\Psi$ is a local isometry, namely, the pull-back ${\Psi}^{\ast}(d{\sigma}^{2})$ 
is equal to the standard Euclidean metric $ds_{Euc}^{2}$ on  ${\Delta}_{R_{0}}$ and, for a point $a_{0}$ with $|a_{0}|=1$, 
the $\Psi$-image ${\Gamma}_{a_{0}}$ of the curve $L_{a_{0}}=\{w:=a_{0}s\, ;\, 0<s<R_{0}\}$ is divergent in $\Sigma$. 
\end{lemma}

\begin{proof}[{\it Proof of Theorem \ref{The3-2}}]
This is proved by contradiction.  
Suppose that $\nu$ omits four distinct values ${\alpha}_{1}, \ldots, {\alpha}_{4}$. 
For our purpose, we may assume ${\alpha}_{4}=\infty$ and that 
$\Sigma$ is biholomorphic to the unit disk 
because $\Sigma$ can be replaced by its universal covering surface and Theorem \ref{The3-2} is obvious in the case 
where $\Sigma =\C$ by the little Picard theorem. 
We choose some $\eta$ with $0<\eta < 1/4$ and set ${\lambda}:=1/(2-4\eta)$. 
Then $1/2< \lambda <1$ holds. 
Now we define a new metric 
\begin{equation}\label{eq3:Thm3-2}
\displaystyle d{\sigma}^{2}=|{G}'_{z}|^{\frac{2}{1-\lambda}} \biggl( \frac{1}{|{\nu}'_{z}|} 
{\prod}_{j=1}^{3}\biggl( \dfrac{|\nu -{\alpha}_{j}|}{\sqrt{1+|{\alpha}_{j}|^2}}{\biggr)}^{1-\eta} 
\biggr)^{\frac{2\lambda}{1-\lambda}}|dz|^{2} 
\end{equation}
on the set ${\Sigma}':=\{z\in \Sigma\, ; \, {\nu}'_{z}(z)\not=0\}$, where $dG ={G}'_{z}dz$ and ${\nu}'_{z}=d\nu/dz$.  
Take a point $p\in {\Sigma}'$. Since the metric $d{\sigma}^{2}$ is flat on ${\Sigma}'$, 
by Lemma \ref{Lem3-2}, there exists a local isometry $\Psi$ satisfying $\Psi (0)=p$ from a disk ${\Delta}_{R}=\{z\in \C\, ;\, |z|<R\}$ 
$(0<R\leq +\infty)$ with the standard Euclidean metric onto an open neighborhood of $p$ in ${\Sigma}'$ 
with the metric $d{\sigma}^{2}$, such that, for a point $a_{0}$ with $|a_{0}|=1$, the $\Psi$-image ${\Gamma}_{a_{0}}$ of the 
curve $L_{a_{0}}=\{w:=a_{0}s\, ;\, 0\leq s<R \}$ is divergent in ${\Sigma}'$. 
For brevity, we denote the function $\nu\circ \Psi$ on $\Delta_{R}$ by $\nu$ in the followings. 
By Lemma \ref{Lem3-1}, we get 
\begin{equation}\label{eq4:Thm3-2}
\displaystyle R\leq C\frac{1+|\nu (0)|^{2}}{|{\nu}'_{z}(0)|}\prod_{j=1}^{4}|\nu (0),\, {\alpha}_{j}|^{1-\eta}<+\infty\,. 
\end{equation}
Hence 
\begin{equation}
L_{d\sigma}({\Gamma}_{a_{0}})=\int_{{\Gamma}_{a_{0}}}d\sigma= \int_{L_{a_{0}}}ds_{Euc}^{2}=R<+\infty\, ,
\end{equation}
where $L_{d\sigma}({\Gamma}_{a_{0}})$ denotes the length of ${\Gamma}_{a_{0}}$ with respect to the metric $d{\sigma}^{2}$. 

We assume that $\Psi$-image ${\Gamma}_{a_{0}}$ tends to a point $p_{0}\in \Sigma\backslash {\Sigma}'$ as $s\to R$. 
Taking a local complex coordinate $\zeta$ in a neighborhood of $p_{0}$ with $\zeta (p_{0})=0$, we can write 
$d{\sigma}^{2}=|\zeta|^{-2\lambda/(1-\lambda)} w|d\zeta|^{2}$ with some positive smooth 
function $w$.  Since $\lambda/(1-\lambda)>1$, we have 
\[
R=\int_{{\Gamma}_{a_{0}}} d\sigma \geq C'\int_{{\Gamma}_{a_{0}}}\dfrac{|d\zeta|}{|\zeta|^{\lambda/(1-\lambda)}}=+\infty
\]
which contradicts (\ref{eq4:Thm3-2}). Thus ${\Gamma}_{a_{0}}$ diverges outside any compact subset of $\Sigma$ as $s\to R$. 

On the other hand, since $d{\sigma}^{2}=|dz|^{2}$, we obtain by (\ref{eq3:Thm3-2}) 
\begin{equation}\label{eq5:Thm3-2}
\displaystyle |{G}'_{z}|=\biggl(|{\nu}'_{z}|\prod_{j=1}^{3}
\biggl(\frac{\sqrt{1+|{\alpha}_{j}|^{2}}}{|\nu -{\alpha}_{j}|} \biggr)^{1-\eta} \biggr)^{\lambda}\,.
\end{equation}
By Lemma \ref{Lem3-1}, we have 
\begin{eqnarray*}
{\Psi}^{\ast}d\tau &=& \sqrt{2}|{G}'_{z}|\sqrt{1+|\nu|^{2}}|dz| \\
                   &=& \sqrt{2}\displaystyle \biggl( |{\nu}'_{z}|(1+|\nu|^{2})^{1/2\lambda} \prod_{j=1}^{3}
                          \biggl( \frac{\sqrt{1+|{\alpha}_{j}|^{2}}}{|\nu -{\alpha}_{j}|}\biggr)^{1-\eta}\biggr)^{\lambda} |dz| \\
                   &=& \sqrt{2}\displaystyle \biggl( \dfrac{|{\nu}'_{z}|}{(1+|\nu|^{2})\prod_{j=1}^{4} |\nu, {\alpha}_{j}|^{1-\eta}} \biggr)^{\lambda}|dz|  \\
                   &\leq& \sqrt{2}C^{\lambda}\biggl(\dfrac{R}{R^{2}-|z|^{2}} \biggr)^{\lambda}|dz|
\end{eqnarray*}
Thus, if we denote the distance $d(p)$ from a point $p\in \Sigma$ to the boundary of $\Sigma$ 
as the greatest lower bound of the lengths with respect to the metric $d\tau^{2}$ of all divergent paths in $\Sigma$, then we have 
\[
d(p)\leq \int_{{\Gamma}_{a_{0}}} d\tau=\int_{L_{a_{0}}} {\Psi}^{\ast}d\tau= \sqrt{2}C^{\lambda}\int_{L_{a_{0}}} 
\biggl( \frac{R}{R^{2}-|z|^{2}}\biggr)^{\lambda}|dz|\leq \sqrt{2}C^{\lambda}\dfrac{R^{1-\lambda}}{1-\lambda} <+\infty 
\]
because $1/2 < \lambda < 1$. However, it contradicts the assumption that $d\tau^{2}$ is complete. 
\end{proof}

As a corollary of Theorem \ref{The3-2}, we provide a new and simple proof of the uniqueness theorem  
for affine complete improper affine spheres from the viewpoint of the value distribution property. 
\begin{corollary}\label{Cor2-2}
Any affine complete improper affine sphere must be an elliptic paraboloid. 
\end{corollary}
\begin{proof}
Because an improper affine sphere has no singularities, 
the complement of the image of its Lagrangian Gauss map $\nu$ contains at least the circle 
$\{|\nu|=1\}\subset\C\cup \{\infty\}$. Thus, by exchanging roles of $dF$ and $dG$ if necessarily, 
it holds that $|\nu|<1$ , that is, $|dF|< |dG|$. On the other hand, we have 
\[
g=|dG|^{2}-|dF|^{2}< 2(|dF|^{2}+|dG|^{2})=d{\tau}^{2}. 
\]
Thus if an improper affine sphere is affine complete, then it is also weakly complete. 
Hence, by Propotion \ref{Pro3-1} and Theorem \ref{The3-2}, it is an elliptic paraboloid. 
\end{proof}

\section{Value distribution of the ratio of canonical forms for weakly complete flat fronts 
in hyperbolic three-space}

We first summarize here definitions and basic facts on weakly complete flat fronts in ${\H}^{3}$ which 
we shall need. For more details, we refer the reader to \cite{GMM}, \cite{KRUY1}, \cite{KRUY2}, \cite{KUY2} 
and \cite{SUY}.  

Let ${\L}^{4}$ be the Lorentz-Minkowski four-space with inner product of signature $(-,+,+,+)$. 
Then the hyperbolic three-space is given by
\begin{equation}\label{hyperbolic-space}
{\H}^{3}=\{(x_{0}, x_{1}, x_{2}, x_{3})\in {\L}^{4} \, | \, -(x_{0})^{2}+(x_{1})^{2}+(x_{2})^{2}+(x_{3})^{2}=-1, x_{0}>0 \}
\end{equation}
with the induced metric from ${\L}^{4}$, which is a simply connected Riemannian three-manifold with constant 
sectional curvature $-1$. 
Identifying ${\L}^{4}$ with the set of $2\times 2$ Hermitian matrices 
Herm($2$)$=\{X^{\ast}=X\}$ $(X^{\ast}:=\tr{\overline{X}}\,)$ by
\begin{equation}\label{Hermite}
(x_{0}, x_{1}, x_{2}, x_{3}) \longleftrightarrow \left(
\begin{array}{cc}
x_{0}+x_{3} & x_{1}+ix_{2} \\
x_{1}-ix_{2} & x_{0}-x_{3}
\end{array}
\right)
\end{equation}
where $i=\sqrt{-1}$, we can write
\begin{eqnarray}\label{hyperbolicspace}
\H^{3}&=& \{X\in \text{Herm(2)}\,;\, \det{X}=1, \trace{X}>0\} \\
      &=& \{aa^{\ast}\,;\, a\in SL(2,\C)\} \nonumber
\end{eqnarray}
with the metric 
\[
\langle X, Y \rangle = -\frac{1}{2}\trace{(X\widetilde{Y})}, \quad  \langle X, X \rangle =-\det(X)\, ,
\]
where $\widetilde{Y}$ is the cofactor matrix of $Y$. The complex Lie group $ PSL(2,\C):= SL(2,\C)/\{\pm \text{id} \}$ 
acts isometrically on 
$\H^{3}$ by 
\begin{equation}\label{action}
\H^{3} \ni X \longmapsto aXa^{\ast}\, , 
\end{equation}
where $a\in PSL(2,\C)$. 

Let $\Sigma$ be an oriented two-manifold. A smooth map $f\colon \Sigma\to {\H}^{3}$ is called a {\it front} 
if there exists a Legendrian immersion 
\[
L_{f}\colon \Sigma \to T_{1}^{\ast}{\H}^{3}
\]
into the unit cotangent bundle of ${\H}^{3}$ whose projection is $f$. 
Identifying $T_{1}^{\ast}{\H}^{3}$ with the unit tangent bundle $T_{1}{\H}^{3}$, 
we can write $L_{f}=(f, n)$, where $n (p)$ is a unit vector in $T_{f(p)}{\H}^{3}$ such that 
$\langle df(p), n (p) \rangle= 0$ for each $p\in M$. We call $n$ a {\it unit normal vector field} of the front $f$. 
A point $p\in\Sigma$ where $\rank{(df)}_{p}<2$ is called a {\it singularity} or {\it singular point}. 
A point which is not singular is called {\it regular point}, where the first fundamental form is positive definite.

The {\it parallel front} $f_{t}$ of a front $f$ at distance $t$ is given by $f_{t}(p)=\text{Exp}_{f(p)}(tn (p))$, 
where ``$\text{Exp}$'' denotes the exponential map of ${\H}^{3}$. 
In the model for ${\H}^{3}$ as in \eqref{hyperbolic-space}, we can write 
\begin{equation}\label{parallel}
f_{t}=(\cosh{t})f+(\sinh{t})n, \quad {n}_{t}=(\cosh{t})n +(\sinh{t})f\,,
\end{equation}
where ${n}_{t}$ is the unit normal vector field of $f_{t}$. 

Based on the fact that any parallel surface of a flat surface is also flat at regular points, 
we define flat fronts as follows: A front $f\colon \Sigma\to {\H}^{3}$ is called a {\it flat front} 
if, for each $p\in M$, there exists a real number $t\in \R$ such that the parallel front $f_{t}$ is a 
flat immersion at $p$. By definition, $\{f_{t}\}$ forms a family of flat fronts. 
We note that an equivalent definition of flat fronts is that the Gaussian curvature 
of $f$ vanishes at all regular points. However, there exists a case where this definition is not suitable. 
For details, see \cite[Remark 2.2]{KUY2}. 

We assume that $f$ is flat. Then there exists a (unique) complex structure on $\Sigma$ and 
a holomorphic Legendrian immersion 
\begin{equation}\label{Legen-lift}
{\E}_{f}\colon \widetilde{\Sigma}\to SL(2,\C)
\end{equation}
such that $f$ and $L_{f}$ are projections of ${\E}_{f}$, where $\widetilde{\Sigma}$ is 
the universal covering surface of $\Sigma$. 
Here, ${\E}_{f}$ being a holomorphic Legendrian map means that ${\E}^{-1}_{f}d{\E}_{f}$ is off-diagonal (see \cite{GMM}, \cite{KUY1}, 
\cite{KUY2}). 
We call ${\E}_{f}$ the {\it holomorphic Legendrian lift} of $f$. 
The map $f$ and its unit normal vector field $n$ are 
\begin{equation}\label{Legen-map-vec}
f={\E}_{f}{\E}^{\ast}_{f}, \quad n = {\E}_{f}e_{3}{\E}^{\ast}_{f}, \quad e_{3}=\left(
\begin{array}{cc}
1 & 0 \\
0 & -1
\end{array}
\right)\,.
\end{equation}

If we set 
\begin{equation}\label{Legen-form}
{\E}^{-1}_{f}d{\E}_{f}=\left(
\begin{array}{cc}
0      & \theta \\
\omega & 0
\end{array}
\right)\, ,
\end{equation}
the first and second fundamental forms $ds^{2}=\langle df, df \rangle$ and 
$dh^{2}=-\langle df, dn \rangle$ are given by
\begin{eqnarray}\label{Legen-form2}
ds^{2}&=&|\omega+\bar{\theta}|^{2}=Q+\bar{Q}+(|\omega|^{2}+|\theta|^{2}), \quad Q=\omega\theta \\
dh^{2}&=&|\theta|^{2}-|\omega|^{2} \nonumber 
\end{eqnarray}
for holomorphic one-forms $\omega$ and $\theta$ defined on $\widetilde{\Sigma}$, with $|\omega|^{2}$ and $|\theta|^{2}$ 
well-defined on $\Sigma$ itself. 
We call $\omega$ and $\theta$ the {\it canonical forms} of $f$. The holomorphic two-differential $Q$ appearing 
in the $(2,0)$-part of $ds^{2}$ is defined on $\Sigma$, and is called the {\it Hopf differential} of $f$. 
By definition, the umbilic points of $f$ coincide with the zeros of $Q$. Defining a meromorphic function on $\widetilde{\Sigma}$ 
by the ratio of canonical forms 
\begin{equation}\label{sign-rho}
\rho=\dfrac{\theta}{\omega}\,,
\end{equation}
then $|\rho|\colon \Sigma\to [0, +\infty]$ is well-defined on $\Sigma$, and $p\in \Sigma$ is a singular point 
if and only if $|\rho(p)|=1$. 

Note that the $(1, 1)$-part of the first fundamental form 
\begin{equation}\label{eq-Sasakian}
ds^{2}_{1,1}=|\omega|^{2}+|\theta|^{2}
\end{equation}
is positive definite on $\Sigma$ because it is the pull-back of the canonical Hermitian metric of $SL(2, \C)$. 
Moreover, $2ds^{2}_{1,1}$ coincides with the pull-back of the Sasakian metric on $T^{\ast}_{1}{\H}^{3}$ 
by the Legendrian lift $L_{f}$ of $f$ (which is the sum of the first and third fundamental forms in this case, 
see \cite[Section 2]{KUY2} for details). The complex structure on $\Sigma$ is compatible with the conformal 
metric $ds^{2}_{1,1}$. Note that any flat front is orientable (\cite[Theorem B]{KRUY1}). 
In this section, for each flat front $f\colon \Sigma\to {\H}^{3}$, 
we always regard $\Sigma$ as a Riemann surface with this complex structure. 

The two {\it hyperbolic Gauss maps} are defined by 
\begin{equation}\label{def-Gauss-map}
G=\dfrac{E_{11}}{E_{21}},\quad G_{\ast}=\dfrac{E_{12}}{E_{22}},\quad  \text{where}\quad  {\E}_{f}=(E_{ij})\,.
\end{equation}
By identifying the ideal boundary ${\Si}^{2}_{\infty}$ of ${\H}^{3}$ with the Riemann sphere ${\C}\cup\{\infty\}$, 
the geometric meaning of $G$ and $G_{\ast}$ is given as follows (\cite{GMM}, \cite[Appendix A]{KRUY2}, \cite{Ro}): 
The hyperbolic Gauss maps $G$ and $G_{\ast}$ represent the intersection points in ${\Si}^{2}_{\infty}$
for the two oppositely-oriented normal geodesics emanating from $f$. 
In particular, $G$ and $G_{\ast}$ are meromorphic functions on $\Sigma$ and parallel fronts have the same hyperbolic 
Gauss maps. We have already obtained an estimate for the totally ramified value numbers of the hyperbolic Gauss maps of 
complete flat fronts in ${\H}^{3}$ in \cite{Ka4}. This estimate is similar to the case of the Gauss map of 
pseudo-algebraic minimal surfaces in Euclidean four-space (see \cite{Ka2}). 
Let $z$ be a local complex coordinate on $\Sigma$. Then we have the following identities (see \cite{KUY2}): 
\begin{equation}\label{Schwarz-i}
s(\omega)-S(G)=2Q,\quad s(\theta)-S(G_{\ast})=2Q, 
\end{equation}
where $S(G)$ is the Schwarzian derivative of $G$ with respect to $z$ as in 
\begin{equation}\label{Schwarz-d}
S(G)=\biggl\{\biggl(\frac{G''}{G'}{\biggr)}'-\frac{1}{2} \biggl(\frac{G''}{G'}{\biggr)}^{2}\biggr\}dz^{2} \qquad 
\biggl(\, '=\frac{d}{dz}\biggr)\, , 
\end{equation}
and $s(\omega)$ and $s(\theta)$ is the Schwarzian derivative of the integral of $\omega$ and $\theta$, respectively.  

Here, we note on the interchangeability of the canonical forms and the hyperbolic Gauss maps. 
The canonical forms $(\omega, \theta)$ have 
the $U(1)$-ambiguity $(\omega, \theta)\mapsto (e^{is}\omega, e^{-is}\theta)\,(s\in \R),$ which corresponds to 
\begin{equation}\label{equ-U(1)}
{\E}_{f}\longmapsto {\E}_{f}\left(
\begin{array}{cc}
e^{is/2} & 0 \\
0 & e^{-is/2}
\end{array}
\right).\, 
\end{equation}
For a second ambiguity, defining the {\it dual} ${\E}_{f}^{\natural}$ of ${\E}_{f}$ by 
\[
{\E}_{f}^{\natural}={\E}_{f}\left(
\begin{array}{cc}
0 & i \\
i & 0
\end{array}
\right),
\] 
then ${\E}_{f}^{\natural}$ is also Legendrian with $f={\E}_{f}^{\natural}{{\E}_{f}^{\natural}}^{\ast}$. 
The hyperbolic Gauss maps $G^{\natural}$, $G_{\ast}^{\natural}$ and canonical forms ${\omega}^{\natural}$, 
${\theta}^{\natural}$ of ${\E}_{f}^{\natural}$ satisfy 
\[
G^{\natural}=G_{\ast}, \quad G_{\ast}^{\natural}=G, \quad {\omega}^{\natural}=\theta, \quad {\theta}^{\natural}=\omega\,.
\]
Namely, the operation $\natural$ interchanges the roles of $\omega$ and $\theta$ and also $G$ and $G_{\ast}$. 

A flat front $f\colon \Sigma \to {\H}^{3}$ is said to be {\it weakly complete} (resp. {\it of finite type}) if 
the metric $ds^{2}_{1,1}$ as in (\ref{eq-Sasakian}) is complete (resp. of finite total curvature). We note that 
the universal cover of a weakly complete flat front is also weakly complete, but completeness is not preserved 
when lifting to the universal cover. 

\begin{fact}[{\cite[Proposition 3.2]{KRUY1}}]\label{WCF-1}
If a flat front $f\colon \Sigma \to {\H}^{3}$ is weakly complete and of finite type, 
then $\Sigma$ is biholomorphic to ${\overline{\Sigma}}_{\gamma}\backslash \{p_{1},\ldots,p_{k}\}$, 
where ${\overline{\Sigma}}_{\gamma}$ is a closed Riemann surface of genus $\gamma$ and $p_{j}\in {\overline{\Sigma}}_{\gamma}$ 
$(j=1,\ldots,k)$. 
\end{fact}
Each puncture point $p_{j}$ $(j=1,\ldots,k)$ is called an {\it WCF-end} of $f$. We can assume that a 
neighborhood of $p_{j}$ is biholomorphic to the punctured disk 
$\D^{\ast}=\{z\in\C\,;\, 0<|z|<1\}$. 
\begin{fact}[\cite{GMM}, \cite{KUY2}, {\cite[Proposition 3.2]{KRUY1}}]\label{WCF-2}
Let $f\colon \D^{\ast}\to {\H}^{3}$ be a WCF-end of a flat front. Then the canonical forms $\omega$ and $\theta$ 
are expressed 
\[
\omega=z^{\mu}{\hat{\omega}}(z)dz,\quad \theta=z^{{\mu}_{\ast}}{\hat{\theta}}(z)dz, \quad (\mu, {\mu}_{\ast}\in \R,\, 
{\mu}+{\mu}_{\ast}\in \Z),
\]
where $\hat{\omega}$ and $\hat{\theta}$ are holomorphic functions in $z$ which do not vanish at the origin. 
In particular, the function $|\rho|\colon \D^{\ast}\to [0,\infty]$ as in (\ref{sign-rho}) can be extended across 
the end. 
\end{fact}

Here, $|\omega|^{2}$ and $|\theta|^{2}$ are considered as conformal flat metrics on $\D^{\ast}_{\varepsilon}$ for 
sufficiently small $\varepsilon >0$. The real numbers $\mu$ and ${\mu}_{\ast}$ are the order of the metrics 
$|\omega|^{2}$ and $|\theta|^{2}$ at the origin respectively, that is, 
\begin{equation}\label{order1}
\mu={\ord}_{0}{|\omega|^{2}}, \quad {\mu}_{\ast}={\ord}_{0}{|\theta|^{2}}. 
\end{equation}
Since $ds^{2}_{1,1}=|\omega|^{2}+|\theta|^{2}$ is complete at the origin, it holds that 
\begin{equation}\label{order2}
{\text{min}}\{\mu, {\mu}_{\ast}\}={\text{min}}\biggl\{ {\ord}_{0}{|\omega|^{2}}, {\ord}_{0}{|\theta|^{2}}\biggr\}
\leq 1\,.
\end{equation}
for a WCF-end. By (\ref{Legen-form2}), the order of the Hopf differential is 
\begin{equation}\label{ord3}
{\ord}_{0}Q=\mu +{\mu}_{\ast}={\ord}_{0}{|\omega|^{2}}+{\ord}_{0}{|\theta|^{2}}. 
\end{equation}

We call the WCF-end {\it regular} if both $G$ and $G_{\ast}$ have at most poles. Then the following fact holds. 
\begin{fact}[\cite{GMM}, {\cite[Proposition 4.2]{KRUY1}}]\label{WCF-3}
A WCF-end $f\colon \D^{\ast}\to {\H}^{3}$ of a flat front is regular if and only if the Hopf differential has 
a pole of order at most two at the origin, that is, ${\ord}_{0}Q\geq -2$ holds. 
\end{fact}


Now we investigate the value distribution of the ratio of canonical forms 
for weakly complete flat fronts in ${\H}^{3}$. We consider the case where the ratio is constant. 

\begin{proposition}\label{prop4-4}
Let $f\colon \Sigma\to {\H}^{3}$ be a weakly complete flat front. 
If the meromorphic function $\rho$ defined by (\ref{sign-rho}) is constant, 
then $f$ is congruent to a horosphere or a hyperbolic cylinder. Here, a surface equidistance from a geodesic 
is called a hyperbolic cylinder \cite{KUY2}. 
\end{proposition}
\begin{proof}
In general, the function $\rho$ is defined on the universal covering surface $\widetilde{\Sigma}$ of $\Sigma$. 
However, in this case, we can consider that $\rho$ is constant on $\Sigma$. 
Then the metric $ds^{2}_{1,1}$ defined by (\ref{eq-Sasakian}) is represented as 
\begin{equation}\label{Sasaki-rho}
ds^{2}_{1,1}=|\omega|^{2}+|\theta|^{2}=\biggl(1+\biggl|\frac{\theta}{\omega}{\biggr|}^{2}\biggr)|\omega|^{2}
=(1+|\rho|^{2})|\omega|^{2}\,.
\end{equation}
Thus the Gaussian curvature $K_{ds^{2}_{1,1}}$ of $ds^{2}_{1,1}$ vanishes identically on $\Sigma$. 
By Fact \ref{WCF-1}, $\Sigma$ is biholomorphic to a closed Riemann surface of genus $\gamma$ with $k$ points removed, 
that is, $\Sigma={\overline{\Sigma}}_{\gamma}\backslash \{p_{1},\ldots,p_{k}\}$. 
Moreover we obtain the formula (\cite[(3.2)]{KRUY1})
\[
\frac{1}{2\pi}\int_{\Sigma}(-K_{ds^{2}_{1,1}})dA=-\chi(\overline{\Sigma}_{\gamma})-
\sum_{j=1}^{k}{\ord}_{p_{j}}(ds^{2}_{1,1}), 
\] 
where $dA$ denotes the area element of $ds^{2}_{1,1}$ and $\chi(\overline{\Sigma}_{\gamma})$ 
the Euler number of $\overline{\Sigma}_{\gamma}$. 
Since the metric $ds^{2}_{1,1}$ is complete, for each WCF-end $p_{j}$, 
${\ord}_{p_{j}}ds^{2}_{1,1}\leq -1$ holds. Thus, in this case, we get $\gamma=0$ and 
\begin{equation}
\sum_{j=1}^{k}{\ord}_{p_{j}}(ds^{2}_{1,1})=-2\,.
\end{equation}
Since $ds^{2}_{1,1}$ is well-defined on $\Sigma$, we need to consider the following two cases: 
\begin{enumerate}
\item[(a)] The flat front $f$ has two WCF-ends $p$ and $q$, and ${\ord}_{p}{ds^{2}_{1,1}}={\ord}_{q}{ds^{2}_{1,1}}=-1$\, ,
\item[(b)] The flat front $f$ has one WCF-end $p$, and ${\ord}_{p}{ds^{2}_{1,1}}=-2$\,. 
\end{enumerate}
In the case (a), $f$ is congruent to a hyperbolic cylinder. In fact, the WCF-ends are asymptotic 
to a finite cover of a hyperbolic cylinder (\cite{GMM}, \cite{KRUY2}). In the case (b), 
then $\rho\equiv 0$. Because, if not, then it holds that ${\ord}_{p}Q=-4$ by (\ref{ord3}). 
On the other hand, the identities (\ref{Schwarz-i}) 
imply that the WCF-end $p$ is regular. However, by Fact \ref{WCF-3}, it does not occur. 
Hence the Hopf differential $Q=\omega\theta$ also vanishes identically on $\Sigma$, and then $f$ is a horosphere. 
\end{proof}

Applying the same argument as in the proof of Theorem \ref{The3-2} to the ratio $\rho$ of weakly complete 
flat fronts in ${\H}^{3}$,
we give the following result for $\rho$. 

\begin{theorem}\label{Thm4-5}
Let $f\colon \Sigma\to {\H}^{3}$ be a weakly complete flat front and $\rho$ the meromorphic function 
on $\widetilde{\Sigma}$ defined by (\ref{sign-rho}). If $\rho$ is nonconstant, then $\rho$ can omit at most three values. 
\end{theorem}

As a corollary of Theorem \ref{Thm4-5}, 
we can obtain the uniqueness theorem of weakly complete flat surfaces in ${\H}^{3}$.  
Note that Sasaki \cite{Sa}, Volkov and Vladimirova \cite{VV} have 
already obtained the same result for complete flat surfaces in ${\H}^{3}$ (See also \cite[Theorem 3]{GMM}).  

\begin{corollary}\label{Cor4-6}
Any weakly complete flat surface in ${\H}^{3}$ must be congruent to a horosphere or a hyperbolic cylinder. 
\end{corollary}
\begin{proof}
Because a weakly complete flat surface has no singularities, 
the complement of the image of $\rho$ contains at least the circle 
$\{|\rho|=1\}\subset\C\cup \{\infty\}$. From Proposition \ref{prop4-4} and Theorem \ref{Thm4-5}, 
it is a horosphere or a hyperbolic cylinder. 
\end{proof}

\end{document}